\documentclass[a4paper,12pt]{article}

\usepackage[left=2cm,right=2cm, top=2cm,bottom=3cm,bindingoffset=0cm]{geometry}

\usepackage{verbatim}
\usepackage{amsmath}
\usepackage{amsthm}
\usepackage{amssymb}
\usepackage{delarray}
\usepackage{cite}
\usepackage{hyperref}
\usepackage{mathrsfs}
\usepackage{tikz}
\usetikzlibrary{patterns}
\usepackage{caption}
\DeclareCaptionLabelSeparator{dot}{. }
\newcommand{\al}{\alpha}

\newcommand{\ga}{\gamma}

\newcommand{\la}{\lambda}

\newcommand{\eps}{\varepsilon}

\newcommand{\iy}{\infty}

\theoremstyle{plain}

\numberwithin{equation}{section}

\newtheorem{thm}{Theorem}[section]
\newtheorem{lem}[thm]{Lemma}

\theoremstyle{definition}

\newtheorem{example}[thm]{Example}

\newtheorem{ip}[thm]{Inverse Problem}

\theoremstyle{remark}

\newtheorem{remark}[thm]{Remark}

\DeclareMathOperator{\diag}{diag}

 \allowdisplaybreaks

\begin{document}

\begin{center}
{\Large\bf Inverse spectral problems for functional-differential\\[0.2cm] operators with involution}
\\[0.2cm]
{\bf Natalia P. Bondarenko} \\[0.2cm]
\end{center}

\vspace{0.5cm}

{\bf Abstract.} The main goal of this paper is to propose an approach to inverse spectral problems for functional-differential operators (FDO) with involution. For definiteness, we focus on the second-order FDO with involution-reflection. Our approach is based on the reduction of the problem to the matrix form and on the solution of the inverse problem for the matrix Sturm-Liouville operator by developing the method of spectral mappings. The obtained matrix Sturm-Liouville operator contains the weight, which causes qualitative difficulties in the study of the inverse problem. As a result, we show that the considered FDO with involution is uniquely specified by five spectra of certain regular boundary value problems.  

\medskip

{\bf Keywords:} inverse spectral problems; functional-differential operators; involution; matrix Sturm-Liouville operator; uniqueness theorem.

\medskip

{\bf AMS Mathematics Subject Classification (2010):} 34K29 34K06 34K08 34K10 34A55   

\vspace{1cm}

\section{Introduction} \label{sec:intr}

The paper is concerned with functional-differential operators (FDO) with involution. For an operator considered on the segment $[-1, 1]$, involution is a non-identical mapping $\nu \colon [-1, 1] \to [-1, 1]$ such that $\nu(\nu(x)) = x$. In this paper, we confine ourselves to the involution-reflection $\nu(x) = -x$. 
In investigation of some aspects, the case of arbitrary smooth involution can be reduced to the case of the reflection.

Ordinary and partial FDO with involution attract attention of scholars because of several reasons. First, non-local equations, which contain the values of the unknown function and its derivatives not only at $x$ but also at some other points, often appear to be more adequate for description of various real-world processes than local differential equations. For example, non-local FDO are widely applied in biological population models and in studying of diffusion processes (see \cite{Wu96}). FDO with reflection have applications in supersymmetric quantum mechanics (see \cite{GPZ99, PVZ11}). Second, differential equations with reflection of argument are used in the investigation of stability of differential-difference equations (see \cite{Shar98}). In addition, FDO with involution cause purely mathematical interest from the both analytic and algebraic points of view (see, e.g., \cite{CT15}).

In recent years, the spectral theory of ordinary FDO with involution has been actively developed. Vladykina and Shkalikov \cite{VSh19-1, VSh19-2} have described the classes of regular boundary conditions and studied spectral properties for the arbitrary-order FDO of form
\begin{equation} \label{hofdo}
y^{(n)}(-x) + \sum_{j = 0}^{n-1} p_j(x) y^{(j)}(-x) + \alpha y^{(m)}(x) + \sum_{j = 0}^{m-1} q_j(x) y^{(j)}(x), \quad \quad x \in (-1, 1).
\end{equation}
There is also a number of studies devoted to the first-order (see \cite{BKLKh07, KKh08, KLS12, BKh14, BKR17, BU18, BKU20, Bur21}) and second-order FDO with involution (see \cite{Sar10, KS12, SS12, KS15, KS17, KSS19, Pol20}). The majority of results of the mentioned papers are concerned with the basis property of eigenfunctions, eigenvalue asymptotics, and other issues of \textit{direct} spectral theory. However, as far as the author knows, there have been no results so far on \textit{inverse} spectral problems which consists in the recovery of FDO with involution from their spectral characteristics. This paper aims to make the first steps in this direction. Our goals are to find spectral characteristics sufficient for the unique reconstruction of an FDO and to propose an approach for further investigation of the corresponding inverse problems.

The greatest success in the theory of inverse spectral problems has been achieved for ordinary differential operators (see the monographs \cite{Mar77, Lev84, PT87, FY01}). Furthermore, inverse problems were studied for certain classes of non-local operators, namely, for integro-differential operators, FDO with frozen argument and with constant delay (see, e.g., the recent papers \cite{BK20, But21, DB21} and the references therein). For those non-local operators, classical methods of inverse problem theory do not work and, therefore, new specific methods were developed. Nevertheless, the situation for FDO with involution-reflection is different. The latter operators can be transformed into differential systems, and the method of spectral mappings \cite{Yur02} can be modified to deal with the corresponding inverse problems. As far as the author knows, this was first noticed by Sergey Buterin.

For definiteness, we focus on the following boundary value problem for the second-order functional-differential equation:
\begin{gather} \label{eqvu}
    -\al u''(x) - u''(-x) + p(x) u(x) + q(x) u(-x) = \la u(x), \quad x \in (-1, 1), \\ \label{bcu}
    u(-1) = u(1) = 0, 
\end{gather}
where $\la$ is the spectral parameter, $\al \in (-1, 1) \cup (\mathbb C \setminus \mathbb R)$, $p, q \in L_1(-1, 1)$. We aim to answer the question: what spectral data uniquely specify the functions $p(x)$ and $q(x)$? For this purpose, we represent the eigenvalue problem~\eqref{eqvu}-\eqref{bcu} in the matrix form:
\begin{gather} \label{eqv}
-Y'' + Q(x) Y = \la W Y, \quad x \in (0, 1), \\ \label{bc}
Y(0) = 0, \quad V(Y) := T Y'(1) - T^{\perp} Y(1) = 0,
\end{gather}
where 
\begin{itemize}
\item $Y = Y(x)$ is a vector function of length $m$ (in our case, $m = 2$);
\item $Q(x)$ is an $(m \times m)$-matrix function, called the \textit{potential}, with complex-valued elements of class $L_1(0, 1)$;
\item $W = \diag\{ w_j \}_{j = 1}^m$ is a diagonal constant matrix, called the \textit{weight};
\item $T$ is an $(m \times m)$ orthogonal projector: $T = T^{\dagger} = T^2$, the symbol ``$\dagger$'' denotes the conjugate transpose.
\item $T^{\perp}$ is the complementary projector: $T^{\perp} = I - T$, $I$ is the unit $(m \times m)$-matrix.
\end{itemize}

Inverse spectral problems for the matrix Sturm-Liouville equation \eqref{eqv} with $W = I$ have been studied fairly completely (see \cite{Mal05, Yur06, CK09, MT09, Bond19, Xu19, Bond21, Bond21-char}). Those results generalize the classical inverse problem theory for the scalar Sturm-Liouville equation $-y'' + q(x) y = \la y$ (see \cite{Mar77, Lev84, PT87, FY01}). However, the case $W \ne I$ causes qualitative difficulties, which are similar, roughly speaking, to the differences in the study of the higher-order differential operators 
\begin{equation} \label{ho}
y^{(n)}(x) + \sum_{k = 0}^{n-2} p_k(x) y^{(k)}(x), \quad n > 2,
\end{equation}
comparing with the second-order ones. Namely, the complex plane of the spectral parameter should be divided into sectors and the asymptotic behavior of the differential equation solutions should be analyzed separately in each sector. The inverse spectral theory of the higher-order differential operators \eqref{ho} has been created by Yurko \cite{Yur02}. Using some ideas of Leibenson \cite{Leib66, Leib71}, Yurko has developed the so-called \textit{method of spectral mappings}. This method allows to prove uniqueness theorems and to obtain constructive procedures for solution together with necessary and sufficient conditions of solvability for various classes of inverse spectral problems. 

In this paper, we develop the ideas of the method of spectral mappings for the matrix Sturm-Liouville problem \eqref{eqv}-\eqref{bc}, and, consequently, obtain the results for the FDO with involution. We prove that the potential $Q(x)$ is uniquely specified by the Weyl matrix, which is a standard spectral characteristic for the matrix Sturm-Liouville operators. This implies the unique specification of the coefficients $p(x)$ and $q(x)$ of equation~\eqref{eqvu} by the spectra of five regular eigenvalue problems for equation \eqref{eqvu} with different boundary conditions. Our results can be applied for developing reconstruction algorithms and for investigating the solvability of the considered inverse problems.

Note that the second-order FDO \eqref{eqvu}-\eqref{bcu} is chosen for simplicity. Our approach can be also applied to the higher-order FDO of form \eqref{hofdo}, but the analysis will be more technically complicated. Other types of the second-order FDO with involution, different from \eqref{eqvu}-\eqref{bcu}, are discussed in Section~\ref{sec:inv}.

It is worth mentioning that the matrix Sturm-Liouville equation \eqref{eqv} can be transformed into the first-order system
$$
\mathscr Q_0 \mathscr Y'(x) + \mathscr Q(x) \mathscr Y(x) = \mu \mathscr Y(x),
$$
where $\mathscr Y(x)$ is a column vector of length $2m$, $\mathscr Q_0$ and $\mathscr Q(x)$ are $(2m \times 2m)$ matrices,
$\mathscr Q_0 = \diag\{ q_k \}_{k = 1}^{2m}$ with distinct non-zero entries $q_k$, $\mu$ is the new spectral parameter. Inverse problems for such systems were also studied by Yurko \cite{Yur05-1, Yur05-2}. However, those results are not applicable for our purposes because of the two reasons. First, the reduction of the second-order system \eqref{eqv}-\eqref{bc} to the first-order one is non-unique. Therefore, the potential $\mathscr Q(x)$ of the first-order system will not be uniquely determined by the spectral data of the second-order system. Second, the boundary conditions from \cite{Yur05-1, Yur05-2} are unsuitable for investigation of FDO with involution. These two issues are discussed in more details in Appendix. Here, we only mention that the specific boundary conditions induced by the continuity of $u(x)$ and $u'(x)$ at $x = 0$ imply the absence of ``the triangular structure'' which appears in \cite{Yur05-1, Yur05-2}, and this is an essential feature of our problem.

The paper is organized as follows. Section~\ref{sec:prelim} contains preliminaries, namely, the reduction of the problem \eqref{eqvu}-\eqref{bcu} to the matrix form \eqref{eqv}-\eqref{bc} and the asymptotic properties of solutions of equation~\eqref{eqv}. In Section~\ref{sec:matr}, we state the inverse problem for the matrix Sturm-Liouville operator \eqref{eqv}-\eqref{bc} and prove the uniqueness of solution for that inverse problem. In Section~\ref{sec:inv}, the results of Section~\ref{sec:matr} are applied to the problem \eqref{eqvu}-\eqref{bcu} with involution. It is shown that $p(x)$ and $q(x)$ are uniquely specified by five spectra of suitable eigenvalue problems. Other types of the second-order FDO with involution are also discussed. In Appendix, we transform the matrix Sturm-Liouville problem~\eqref{eqv}-\eqref{bc} into the first-order system and discuss the difference of our inverse problem from the inverse problem for the first-order system. 

\section{Preliminaries} \label{sec:prelim}

\subsection{Reduction to the matrix form}

Obviously, equation~\eqref{eqvu} can be represented in the matrix form
\begin{equation} \label{eqvZ}
    -Z'' + \mathcal Q(x) Z = \la \mathcal W Z, \quad x \in (0, 1),
\end{equation}
where
\begin{equation} \label{changeZ}
Z(x) = \begin{bmatrix} z_1(x) \\ z_2(x) \end{bmatrix} = \begin{bmatrix} u(-x) \\ u(x) \end{bmatrix}, \quad
\mathcal W = \begin{bmatrix} \al & 1 \\ 1 & \al \end{bmatrix}^{-1},
\quad \mathcal Q(x) = \mathcal W\begin{bmatrix} p(-x) & q(-x) \\ q(x) & p(x) \end{bmatrix}.
\end{equation}
The continuity of $u(x)$ and $u'(x)$ at zero implies
$$
z_1(0) = z_2(0), \quad z_1'(0) + z_2'(0) = 0.
$$
This together with \eqref{bcu} yield the boundary conditions
\begin{equation} \label{bcZ}
\mathcal V(Z) := \mathcal T Z'(0) + \mathcal T^{\perp} Z(0) = 0, \quad Z(1) = 0,
\end{equation}
where $\mathcal T$ and $\mathcal T^{\perp}$ are the following orthogonal projectors:
$$
\mathcal T = \frac{1}{2} \begin{bmatrix} 1 & 1 \\ 1 & 1\end{bmatrix}, \quad \mathcal T^{\perp} = I - \mathcal T = \frac{1}{2} \begin{bmatrix} 1 & -1 \\ -1 & 1\end{bmatrix}.
$$
Diagonalizing the matrix $\mathcal W$ and changing $x \to 1 - x$, we reduce the eigenvalue problem \eqref{eqvZ},\eqref{bcZ} to form~\eqref{eqv}-\eqref{bc}, where
\begin{equation} \label{YZ}
\begin{array}{c}
Y(x) = U Z(1-x), \quad Q(x) = U \mathcal Q(1 - x) U^{\dagger}, \quad W = U \mathcal W U^{\dagger}, \quad T = U \mathcal T U^{\dagger}, \\
U = \frac{1}{\sqrt 2} \begin{bmatrix} 1 & 1 \\ -1 & 1 \end{bmatrix}, \quad W = \begin{bmatrix} \frac{1}{\al + 1} & 0 \\ 0 & \frac{1}{\al - 1}\end{bmatrix}, \quad T = \begin{bmatrix} 1 & 0 \\ 0 & 0 \end{bmatrix}, \quad T^{\perp} = \begin{bmatrix} 0 & 0 \\ 0 & 1 \end{bmatrix}.
\end{array}
\end{equation}
One can easily check that
$$
\al \in (-1,1) \cup (\mathbb C \setminus \mathbb R) \quad \Leftrightarrow \quad w_1 w_2 \ne 0, \quad \arg w_1 \ne \arg w_2,
$$
where $w_j = \frac{1}{\al + (-1)^{j + 1}}$, $W = \diag\{ w_1, w_2 \}$.

\subsection{Asymptotics of solutions} \label{sec:asympt}

In this subsection, we study the asymptotics of special solutions of
the matrix equation~\eqref{eqv} with an arbitrary complex-valued weight $W = \diag \{ w_k \}_{k = 1}^m$ such that $w_k \ne 0$, $w_k \ne w_j$ for $k \ne j$.

Put $\la = \rho^2$.
We start with the construction of the Birkhoff solutions with the certain asymptotic behavior as $|\rho| \to \iy$, by using the standard method (see, e.g., \cite{Nai68}). The complex plane $\mathbb C$ can be divided into sectors $\Gamma_j = \{ \rho \colon \arg \rho \in (\theta_{j-1}, \theta_j)\}$, $j = \overline{1, r}$, $0 \le \theta_0 < \theta_1 < \ldots < \theta_{r-1} < \theta_r = \theta_0 + 2\pi$ so that, in each fixed sector $\Gamma = \Gamma_j$, the numbers $\pm i \sqrt w_k$, $k = \overline{1, m}$ can be renumbered as $R_1$, $R_2$, \ldots, $R_{2m}$ so that
$$
\mbox{Re} \, (\rho R_1) < \mbox{Re} \, (\rho R_2) < \dots < \mbox{Re} \, (\rho R_{2m}), \quad \rho \in \Gamma.
$$

For each fixed sector $\Gamma$, equation~\eqref{eqv} with $\la = \rho^2$ has the fundamental system of vector solutions $\{ E_k(x, \rho) \}_{k = 1}^{2m}$ defined by the following integral equations:
\begin{multline*}
    E_k(x, \rho) = f_k \exp(\rho R_k x) + \sum_{j = 1}^k \frac{1}{2 \rho R_j} \int_0^x J_j \exp(\rho R_j (x - t)) Q(t) E_k(t, \rho) \, dt \\ + \sum_{j = k + 1}^{2m} \frac{1}{2\rho R_j} \int_x^1 J_j \exp(\rho R_j(x - t)) Q(t) E_k(t, \rho) \, dt, \quad k = \overline{1, 2m},
\end{multline*}
where $f_j = [f_{j,s}]_{s = \overline{1,m}}^T$ is the column vector, $f_{j,s} = 1$ if $\pm i\sqrt{w_s} = R_j$ and $f_{j,s} = 0$ otherwise, $J_j = f_j f_j^{\dagger}$. The solutions $\{ E_k(x, \rho) \}_{k = 1}^{2m}$ have the following properties for $\nu = 0, 1$ and some $\rho^* > 0$:

\smallskip

(i) The vector functions $E_k^{(\nu)}(x, \rho)$ are continuous for $x \in [0, 1]$, $\rho \in \overline{\Gamma}$, $|\rho| \ge \rho^*$;

(ii) For each fixed $x \in [0, 1]$, the vector functions $E_k^{(\nu)}(x, \rho)$ are analytic with respect to $\rho \in \Gamma$, $|\rho| \ge \rho^*$;

(iii) $E_k^{(\nu)}(x, \rho) = (\rho R_k)^{\nu} \exp(\rho R_k x) (f_k + O(\rho^{-1}))$ as $|\rho| \to \iy$, $\rho \in \overline{\Gamma}$, uniformly with respect to $x \in [0, 1]$.

\smallskip

For a fixed sector $\Gamma$, denote $d_k = \sqrt w_k$, $k = \overline{1, m}$, choosing the sign of the square root so that $\mbox{Re}(\rho d_k) > 0$ for $\rho \in \Gamma$. Set $D = \diag\{ d_k \}_{k = 1}^m$. By using the columns $\{ E_k(x, \rho) \}_{k = 1}^{2m}$, one can form the $(m \times m)$ matrix solutions $E_{\pm}(x, \rho)$ of equation~\eqref{eqv} with the properties similar to (i)-(ii) and satisfying the asymptotic relation
\begin{equation} \label{asymptE}
E_{\pm}^{(\nu)}(x, \rho) = [I] (\pm i \rho D)^{\nu}\exp(\pm i \rho D x), \quad |\rho| \to \iy, \quad \rho \in \overline{\Gamma}, \quad \nu = 0, 1,
\end{equation}
uniformly with respect to $x \in [0, 1]$.
Here and below $[I] = I + O(\rho^{-1})$.

Denote by $C(x, \la)$ and $S(x, \la)$ the matrix solutions of equation~\eqref{eqv} under the initial conditions
\begin{equation} \label{icCS}
C(0, \la) = S'(0, \la) = I, \quad C'(0, \la) = S(0, \la) = 0.
\end{equation}
Clearly, the matrix functions $C^{(\nu)}(x, \la)$, $S^{(\nu)}(x, \la)$ are entire in $\la$ for each fixed $x \in [0, 1]$ and $\nu = 0, 1$.
Using \eqref{asymptE} and~\eqref{icCS}, one can easily show that
\begin{align} \label{asymptC}
& C^{(\nu)}(x, \la) = \frac{1}{2} (E_+^{(\nu)}(x, \rho)[I] + E_-^{(\nu)}(x, \rho)[I]), \\ \label{asymptS}
& S^{(\nu)}(x, \la) = \frac{1}{2} (i \rho D)^{-1} (E_+^{(\nu)}(x, \rho)[I] - E_-^{(\nu)}(x, \rho)[I]), 
\end{align}
for $|\rho| \to \iy$, $\rho \in \overline{\Gamma}$,
$\nu = 0, 1$, $\la = \rho^2$, uniformly with respect to $x \in [0, 1]$.

\section{Matrix Sturm-Liouville operator} \label{sec:matr}

Consider the boundary value problem $L = L(Q, W, T)$ of form~\eqref{eqv}-\eqref{bc} with $m = 2$, a matrix potential $Q(x)$ with entries of class $L_1(0, 1)$, a weight $W = \diag\{w_1, w_2\}$ such that $\arg w_1 \ne \arg w_2$, $w_1 w_2 \ne 0$, and an arbitrary orthogonal projector $T$.

In this section, we formulate the inverse spectral problem for the matrix Sturm-Liouville operator, and prove the uniqueness theorem for this inverse problem. The presence of the weight $W \ne I$ causes the following difficulty. The asymtotics of solutions of equation \eqref{eqv} contain the two exponents $\exp(i \rho d_k x)$, $k = 1, 2$, with $d_1 \ne d_2$. Therefore, we find such rays $\arg \rho = \theta_s$ that $\mbox{Re}(i \rho d_1) = \mbox{Re}(i \rho d_2)$. The asymptotic behavior of the solutions along these rays allows us to apply the method of spectral mappings to the inverse problem.

The eigenvalues of $L$ coincide with the zeros of the characteristic function $\Delta(\la) = \det(V(S(x, \la)))$, which is entire in $\la$. Taking the asymptotics \eqref{asymptS} into account, one can show that $\Delta(\la)$ has a countable set of the zeros $\{ \la_n \}$ (see Section~\ref{sec:inv} for details).

Denote by $\Phi(x, \la)$ the matrix solution of equation~\eqref{eqv} satisfying the boundary conditions 
\begin{equation} \label{bcPhi}
\Phi(0, \la) = I, \quad V(\Phi) = 0,
\end{equation}
and put $M(\la) = \Phi'(0, \la)$. The matrix functions $\Phi(x, \la)$ and $M(\la)$ are called the \textit{Weyl solution} and the \textit{Weyl matrix} of the problem $L$, respectively. The notion of Weyl matrix generalizes the notion of Weyl function for the scalar Sturm-Liouville operator (see \cite{Mar77, FY01}). In the scalar case, the specification of the Weyl function is equivalent to the specification of the two spectra of Borg's problem \cite{Borg46}. Weyl functions and their generalizations are natural spectral characteristics in the inverse problem theory. In particular, the Weyl matrices have been used for reconstruction of the matrix Sturm-Liouville operators with $W = I$ in \cite{CK09, MT09, Bond19, Xu19, Bond21, Bond21-char}.

One can easily obtain the following relations
\begin{gather} \label{relPhi}
    \Phi(x, \la) = C(x, \la) + S(x, \la) M(\la), \\ \label{relM}
    M(\la) = -(V(S))^{-1} V(C).
\end{gather}
Consequently, the matrix functions $M(\la)$ and $\Phi(x, \la)$ for each fixed $x \in [0, 1]$ are meromorphic in the $\la$-plane, and their poles coincide with the eigenvalues of $L$. 

Proceed with the investigation of the Weyl solution asymptotics.
Consider the partition of the complex plane into sectors described in Subsection~\ref{sec:asympt}. Under the condition $\arg w_1 \ne \arg w_2$, there exist the rays $\arg \rho = \theta_{j_s}$, $s = \overline{1, 4}$, $0 \le \theta_{j_1} < \theta_{j_2} < \pi$, $\theta_{j_3} = \theta_{j_1} + \pi$, $\theta_{j_4} = \theta_{j_2} + \pi$, such that
\begin{equation} \label{ray}
\mbox{Re} (i \rho d_1) = \mbox{Re} (i \rho d_2) > 0, \quad \arg \rho = \theta_{j_s}.
\end{equation}
Note that, although the choice of the square root sign $d_j = \pm \sqrt w_j$ depends on the sector $\Gamma$, in the two neighboring sectors separated by the ray $\arg \rho = \theta_{j_s}$, the choice of the signs is the same.

\begin{example} \label{ex:sectors}
Consider the case $w_1 = 1$, $w_2 = -1$ corresponding to the operator \eqref{eqvu}-\eqref{bcu} with involution and $\al = 0$. Then $\theta_j = \frac{\pi j}{4}$, $j = \overline{0, 8}$, $\{ j_1, j_2, j_3, j_4 \} = \{ 1, 3, 5, 7 \}$. In particular, in the sectors $\Gamma_1 = \{ \rho \colon \arg \rho \in (0, \tfrac{\pi}{4}) \}$ and $\Gamma_2 = \{ \rho \colon \arg \rho \in (\tfrac{\pi}{4}, \tfrac{\pi}{2}) \}$, we have $d_1 = -1$, $d_2 = -i$, and
\begin{align*}
& \mbox{Re}(-i\rho d_2) < \mbox{Re}(-i\rho d_1) < 0 < \mbox{Re}(i \rho d_1) < \mbox{Re}(i \rho d_2), \quad \rho \in \Gamma_1, \\
& \mbox{Re}(-i\rho d_1) < \mbox{Re}(-i\rho d_2) < 0 < \mbox{Re}(i \rho d_2) < \mbox{Re}(i \rho d_1), \quad \rho \in \Gamma_2. 
\end{align*}
In other words, when $\rho$ passes over the ray $\arg \rho = \tfrac{\pi}{4}$, the numbers $d_1$ and $d_2$ remain the same, but the values $\mbox{Re}(i\rho d_1)$ and $\mbox{Re}(i\rho d_2)$ change their order. The situation in the other quarter-planes is symmetric.

\begin{figure}[h!]
\begin{center}
\begin{tikzpicture}
\fill[fill=gray!20] (0, 0)--(2, 0)--(1.6, 1.6)--(0, 2);
\draw (-2, 0) edge (2, 0);
\draw (0, -2) edge (0, 2);
\draw[thick] (-1.6, -1.6) edge (1.6, 1.6);
\draw[thick] (1.6, -1.6) edge (-1.6, 1.6);
\draw (1.2, 0.5) node{$\Gamma_1$};
\draw (0.5, 1.2) node{$\Gamma_2$};
\draw (1.8, 1.8) node{$\theta_1$};
\draw (-1.8, 1.8) node{$\theta_3$};
\draw (-1.8, -1.8) node{$\theta_5$};
\draw (1.8, -1.8) node{$\theta_7$};
\end{tikzpicture}
\end{center}
\caption{Sectors in Example~\ref{ex:sectors}}
\label{fig:1}
\end{figure}

\end{example}

\begin{lem} \label{lem:asympt}
The following asymptotic relations hold
\begin{gather} \label{asymptCS}
    C^{(\nu)}(x, \la) = \frac{1}{2} (i \rho D)^{\nu} \exp(i \rho D x)[I], \quad 
    S^{(\nu)}(x, \la) = \frac{1}{2} (i \rho D)^{\nu - 1} \exp(i \rho D x)[I], \quad x \in (0, 1], \\ \label{asymptPhi}
    \Phi^{(\nu)}(x, \la) = \exp(-i\rho Dx)[I], \quad x \in [0, 1),
\end{gather}
as $|\rho| \to \iy$, $\arg \rho =\theta_{j_s}$, $s = \overline{1, 4}$, $\la = \rho^2$, $\nu = 0, 1$.
\end{lem}

\begin{proof}
Observe that condition \eqref{ray} implies 
\begin{equation} \label{change}
\exp(i \rho D x) [I] = [I] \exp(i \rho D x), \quad \arg \rho = \theta_{j_s}, \: s = \overline{1, 4}.
\end{equation}
Combining \eqref{asymptE}, \eqref{asymptC}-\eqref{asymptS}, and \eqref{change}, we easily obtain \eqref{asymptCS}. 

Fix a ray $\arg \rho = \theta_{j_s}$ and an adjacent sector $\Gamma$. Consider the solutions $E_{\pm}(x, \rho)$, $\rho \in \overline{\Gamma}$. Clearly, the $2m$ columns of the matrix functions $E_{\pm}(x, \rho)$ form a fundamental system of solutions of equation~\eqref{eqv} on the ray $\arg \rho = \theta_{j_s}$, so the Weyl solution can be expanded with respect to this system:
\begin{equation} \label{expandPhi}
\Phi(x, \la) = E_+(x, \rho) B_+(\rho) + E_-(x, \rho) B_-(\rho), \quad \arg \rho = \theta_{j_s}, \quad \la = \rho^2,
\end{equation}
where $B_{\pm}(\rho)$ are $(m \times m)$ matrix functions.
Using \eqref{asymptE}, \eqref{bcPhi}, and \eqref{change}, we show that
\begin{equation} \label{asymptB}
B_+(\rho) = O(\exp(-2 \mbox{Re}(i \rho D))), \quad B_-(\rho) = [I].
\end{equation}
Substituting \eqref{asymptE}, \eqref{asymptB} into \eqref{expandPhi} and using~\eqref{change}, we arrive at \eqref{asymptPhi}.
\end{proof}

Along with \eqref{eqv}, consider the equation
\begin{gather} \label{eqZ2}
    -Z'' + Z Q(x) = \la Z W, \quad x \in (0, 1).
\end{gather}
Let $Y(x)$ and $Z(x)$ be arbitrary solutions of equations \eqref{eqv} and \eqref{eqZ2}, respectively. Define the matrix Wronskian $\langle Z, Y \rangle := Z Y' - Z' Y$. Calculations imply the relation
\begin{equation} \label{wron}
\frac{d}{dx}\langle Z, Y \rangle = 0.
\end{equation}

Denote by $C^*(x, \la)$, $S^*(x, \la)$, and $\Phi^*(x, \la)$ the solutions of equation~\eqref{eqZ2} satisfying the conditions
\begin{gather} \label{icCS*}
C^*(0, \la) = {S^*}'(0, \la) = I, \quad {C^*}'(0, \la) = S^*(0, \la) = 0, \\ \label{bcPhi*}
\Phi^*(0, \la) = I, \quad V^*(\Phi^*) := {\Phi^*}'(1, \la) T - \Phi^*(1, \la) T^{\perp} = 0.
\end{gather}
Denote $M^*(\la) := {\Phi^*}'(0, \la)$. Similarly to \eqref{relPhi}, we get
\begin{equation} \label{relPhi*}
\Phi^*(x, \la) = C^*(x, \la) + M^*(\la) S^*(x, \la).
\end{equation}
Using \eqref{wron}, \eqref{bcPhi}, and \eqref{bcPhi*}, we obtain
\begin{align*}
\langle \Phi^*, \Phi \rangle = \langle \Phi^*, \Phi \rangle_{|x = 0} = & M(\la) - M^*(\la), \\
\langle \Phi^*, \Phi \rangle = \langle \Phi^*, \Phi \rangle_{|x = 1} = & \Phi^*(1, \la) T \Phi'(1, \la) + \Phi^*(1, \la) T^{\perp} \Phi'(1, \la) \\ & - {\Phi^*}'(1, \la) T \Phi(1, \la) - {\Phi^*}'(1, \la) T^{\perp} \Phi(1, \la) = 0.
\end{align*}
Hence $M(\la) \equiv M^*(\la)$.

Let us study the following inverse problem. The matrices $W$ and $T$ are supposed to be known.

\begin{ip} \label{ip:matr}
Given the Weyl matrix $M(\la)$, find $Q(x)$.
\end{ip}

Along with the boundary value problem $L = L(Q, W, T)$ consider another problem $\tilde L = L(\tilde Q, W, T)$ of the same form but with another potential $\tilde Q(x)$. The matrices $W$ and $T$ for these two boundary value problems coincide. If a symbol $\ga$ denotes an object related to the problem $L$, then the symbol $\tilde \ga$ with tilde will be used for the analogous object related to $\tilde L$. The following theorem asserts the uniqueness of solution for Inverse Problem~\ref{ip:matr}.

\begin{thm} \label{thm:uniqm}
If $M(\la) \equiv \tilde M(\la)$, then $Q(x) = \tilde Q(x)$ a.e. on $(0, 1)$.
\end{thm}

\begin{proof}
Define the $(2m \times 2m)$ matrix of spectral mappings $P(x, \la) = [P_{jk}(x, \la)]_{j,k=1,2}$, consisting of the $(m \times m)$ matrix blocks $P_{jk}(x, \la)$, by the following relation:
\begin{equation} \label{defP}
\begin{bmatrix}
P_{11} & P_{12} \\ P_{21} & P_{22}
\end{bmatrix} 
\begin{bmatrix}
\tilde S & \tilde \Phi \\ \tilde S' & \tilde \Phi'
\end{bmatrix} =
\begin{bmatrix}
S & \Phi \\ S' & \Phi'
\end{bmatrix}. 
\end{equation}
Here and below in this proof, the arguments $(x, \la)$ are omitted for brevity. Using \eqref{wron} and the conditions for $\tilde S$, $\tilde \Phi$, $\tilde S^*$, $\tilde \Phi^*$ at $x = 0$, we get
\begin{equation*} 
\begin{bmatrix}
\tilde S & \tilde \Phi \\ \tilde S' & \tilde \Phi'
\end{bmatrix}^{-1} = 
\begin{bmatrix}
-\tilde \Phi^*{}' & \tilde \Phi^* \\ \tilde S^*{}' & -\tilde S^*
\end{bmatrix}.
\end{equation*}
Consequently,
$$
\begin{bmatrix}
P_{11} & P_{12} \\ P_{21} & P_{22}
\end{bmatrix} 
= \begin{bmatrix}
S & \Phi \\ S' & \Phi'
\end{bmatrix}
\begin{bmatrix}
-\tilde \Phi^*{}' & \tilde \Phi^* \\ \tilde S^*{}' & -\tilde S^*
\end{bmatrix}.
$$
In particular,
\begin{equation} \label{P1j}
P_{11} = -S \tilde \Phi^*{}' + \Phi \tilde S^*{}', \quad P_{12} = S \tilde \Phi^* - \Phi \tilde S^*.
\end{equation}
Using the asymptotics of Lemma~\ref{lem:asympt} for $S$, $\Phi$ and the similar asymptotics for $\tilde S^*$, $\tilde \Phi^*$, we get
\begin{equation} \label{asymptP}
P_{11}(x, \rho^2) = [I], \quad P_{12}(x, \rho^2) = O(\rho^{-1}), \qquad |\rho| \to \iy,
\end{equation}
for $\arg \rho = \theta_{j_s}$, $s = \overline{1, 4}$, and each fixed $x \in (0, 1)$. 

On the other hand, the substitution of \eqref{relPhi} and \eqref{relPhi*} into \eqref{P1j} yields
\begin{align*}
& P_{11} = -S \tilde C^*{}' + C \tilde S^*{}' + S(M - \tilde M^*)\tilde S^*{}', \\
& P_{12} = S \tilde C^* - C \tilde S^* + S(\tilde M^* - M) \tilde S^*.
\end{align*}
Since $M(\la) \equiv \tilde M(\la) \equiv \tilde M^*(\la)$, then the matrix functions $P_{11}(x, \rho^2)$ and $P_{12}(x, \rho^2)$ are entire in $\rho$ of order $1$ for each fixed $x \in (0, 1)$. Since \eqref{asymptP} hold on the rays $\arg \rho = \theta_{j_k}$ and the angles between the neighboring rays are less than $\pi$, then Phragmen-Lindel\"of's theorem (see \cite{BFY14}) implies the asymptotics \eqref{asymptP} in the whole $\rho$-plane. Consequently, by virtue of Liouville's theorem, $P_{11}(x, \la) \equiv I$, $P_{12}(x, \la) \equiv 0$. Therefore, \eqref{defP} implies $S(x, \la) \equiv \tilde S(x, \la)$. Hence, $Q(x) = \tilde Q(x)$ a.e. on $(0, 1)$.
\end{proof}

Using the method of spectral mappings, one can obtain a constructive solution of Inverse Problem~\ref{ip:matr}.

\section{Operators with involution} \label{sec:inv}

In the previous section, we have proved the uniqueness of recovering the potential $Q(x)$ of the matrix Sturm-Liouville operator from the Weyl matrix $M(\la)$. The goal of this section is to show that the specification of the Weyl matrix is equivalent to the specification of the five spectra of operators with involution induced by equation~\eqref{eqvu} and different boundary conditions. Consequently, we prove that those five spectra uniquely specify the coefficients $p(x)$ and $q(x)$ of \eqref{eqvu}.

It will be convenient for us to return from the problem~\eqref{eqv}-\eqref{bc} to \eqref{eqvZ},\eqref{bcZ}. Denote by $\mathcal C(x, \la)$ and $\mathcal S(x, \la)$ the matrix solutions of equation~\eqref{eqvZ} satisfying the initial conditions
$$
\mathcal C(1, \la) = \mathcal S'(1, \la) = I, \quad \mathcal C'(1, \la) = \mathcal S(1, \la) = 0.
$$
In view of \eqref{YZ},
$$
\mathcal C(x, \la) = U^{\dagger} C(1-x, \la) U, \quad
\mathcal S(x, \la) = -U^{\dagger} S(1-x, \la) U.
$$
Using \eqref{relM}, we obtain
\begin{equation} \label{UMU}
U^{\dagger} M(\la) U = (\mathcal V(\mathcal S))^{-1} \mathcal V(\mathcal C) =: \mathcal M(\la).
\end{equation}

By using Cramer's rule, we represent $\mathcal M(\la)$ in the form
\begin{equation} \label{MD}
\mathcal M(\la) = \frac{1}{\Delta(\la)} \begin{bmatrix} 
\Delta_{11}(\la) & \Delta_{12}(\la) \\
\Delta_{21}(\la) & \Delta_{22}(\la)
\end{bmatrix},
\end{equation}
where
\begin{gather*}
    \Delta(\la) = \det V(S) = \det \mathcal V(\mathcal S), \\
    \Delta_{11}(\la) := \det [\mathcal V(\mathcal C_1), \mathcal V(\mathcal S_2)], \quad
    \Delta_{12}(\la) := \det [\mathcal V(\mathcal C_2), \mathcal V(\mathcal S_2)], \\
    \Delta_{21}(\la) := \det [\mathcal V(\mathcal S_1), \mathcal V(\mathcal C_1)], \quad
    \Delta_{22}(\la) := \det [\mathcal V(\mathcal S_1), \mathcal V(\mathcal C_2)].    
\end{gather*}
The notations $\mathcal C_1$, $\mathcal C_2$ and $\mathcal S_1$, $\mathcal S_2$ are used for the corresponding columns of the matrix functions $\mathcal C(x, \la)$ and $\mathcal S(x, \la)$, respectively. Obviously, the functions $\Delta(\la)$ and $\Delta_{jk}(\la)$, $j, k = 1, 2$, are entire in $\la$ of order not greater than $\frac{1}{2}$.

Observe that $\Delta(\la)$ is the characteristic function of the eigenvalue problem \eqref{eqvu}-\eqref{bcu} (this problem will be denoted by $\mathcal L$). Similarly, $\Delta_{jk}(\la)$, $j, k = 1, 2$, are the characteristic functions of the corresponding eigenvalue problems $\mathcal L_{jk}$ for equation~\eqref{eqvu} with the following boundary conditions:
\begin{gather*}
    \mathcal L_{11} \colon u'(-1) = u(1) = 0, \qquad
    \mathcal L_{12} \colon u(-1) = u'(-1) = 0, \\
    \mathcal L_{21} \colon u(1) = u'(1) = 0, \qquad
    \mathcal L_{22} \colon u(-1) = u'(1) = 0.
\end{gather*}
That is, the zeros of the characteristic functions coincide with the eigenvalues of the corresponding problems. Note that, although the problems $\mathcal L_{12}$ and $\mathcal L_{21}$ have initial conditions, they are spectral problems, since equation~\eqref{eqvu} is non-local. All the five problems $\mathcal L$, $\mathcal L_{jk}$, $j, k = 1, 2$, are regular in terms of \cite{VSh19-1}.

\begin{lem} \label{lem:asymptD}
The following asymptotic relations hold for $|\rho| \to \iy$, $\arg \rho = \theta_{j_s}$, $s = \overline{1, 4}$:
\begin{gather} \label{asymptD}
    \Delta(\la) = -\frac{1}{2 i \rho d_2} \exp(i \rho (d_1  + d_2))[1], \\ \label{asymptDjk}
    \Delta_{jk}(\la) = \frac{1}{4} \al_{jk} \exp(i \rho (d_1 + d_2))[1], \quad \al_{jk} = \begin{cases} \frac{d_1}{d_2} + 1, \quad j = k, \\ \frac{d_1}{d_2} - 1, \quad j \ne k, \end{cases}
    \quad j, k = 1, 2,
\end{gather}
where $\la = \rho^2$, $[1] = 1 + O(\rho^{-1})$.
\end{lem}

\begin{proof}
In this proof, we suppose that $\arg \rho = \theta_{j_s}$, $s \in \{ 1, 2, 3, 4 \}$, $\la = \rho^2$.
The asymptotics \eqref{asymptCS} imply
\begin{equation} \label{VSC}
V(S) = \frac{1}{2} \Theta(\rho) \exp(i \rho D)[I], \quad
V(C) = \frac{1}{2} \Theta(\rho) i \rho D \exp(i \rho D)[I], \quad |\rho| \to \iy,
\end{equation}
where 
$$
\Theta(\rho) = T - T^{\perp}(i \rho D)^{-1}, \quad \det \Theta(\rho) = -(i \rho d_2)^{-1}.
$$
Therefore, we immediately obtain the asymptotics \eqref{asymptD} for $\Delta(\la) = \det V(S)$.
Using \eqref{relM}, \eqref{UMU}, and \eqref{VSC}, we get
$$
M(\la) = -i\rho D[I], \quad \mathcal M(\la) = -i\rho U^{\dagger} D U [I], \quad |\rho|\to \iy.
$$
Using the latter asymptotics together with \eqref{change}, \eqref{MD}, and \eqref{asymptD}, we arrive at \eqref{asymptDjk}. Note that $\al_{jk}$ defined in \eqref{asymptDjk} are non-zero, because $w_1 \ne w_2$ and so $d_1 \ne \pm d_2$.
\end{proof}

In view of the asymptotics \eqref{asymptD} and \eqref{asymptDjk}, the characteristic functions $\Delta(\la)$ and $\Delta_{jk}(\la)$, $j, k = 1, 2$, have the countable sets of zeros $\Lambda = \{ \la_n \}$ and $\Lambda_{jk} = \{ \la_{n,jk} \}$, respectively, counted with their multiplicities and being the eigenvalues of the corresponding boundary value problems. By Hadamard's factorization theorem, the characteristic functions can be represented in the form
\begin{equation} \label{prod}
\Delta(\la) = c \prod_n \left( 1 - \frac{\la}{\la_n}\right), \quad \Delta_{jk}(\la) = c_{jk} \prod_n \left( 1 - \frac{\la}{\la_{n,jk}}\right), \quad j,k = 1, 2.
\end{equation}
Here, for simplicity, we assume that all the eigenvalues are non-zero. The case of zero eigenvalues requires minor changes. The constants $c$ and $c_{jk}$ in \eqref{prod} can be found by the following formulas, obtained by using Lemma~\ref{lem:asymptD}:
\begin{align} \label{formc}
    c & = -\frac{1}{2 i d_2} \lim_{\substack{|\rho| \to \iy \\ \arg \rho = \theta_{j_s}}} \rho^{-1} \exp(i \rho (d_1 + d_2)) \prod_n \left( 1 - \frac{\rho^2}{\la_n}\right)^{-1}, \\ \label{formcjk}
    c_{jk} & = \frac{1}{4}\al_{jk} \lim_{\substack{|\rho| \to \iy \\ \arg \rho = \theta_{j_s}}} \exp(i \rho (d_1 + d_2)) \prod_n \left( 1 - \frac{\rho^2}{\la_{n,jk}}\right)^{-1}, \quad j, k = 1, 2.
\end{align}

Thus, it is natural to consider the following five-spectra inverse problem.

\begin{ip} \label{ip:5}
Given the spectra $\Lambda$, $\Lambda_{jk}$, $j,k = 1, 2$, of the eigenvalue problems $\mathcal L$, $\mathcal L_{jk}$, $j,k = 1, 2$, respectively, find $p$ and $q$.
\end{ip}

The parameter $\al \in (-1, 1) \cup (\mathbb C \setminus \mathbb R)$ is supposed to be fixed and known a priori.

In the theorem below, we suppose that, if a symbol $\ga$ denotes an object related to equation \eqref{eqvu}, then the symbol $\tilde \ga$ with tilde denotes the analogous object related to equation \eqref{eqvu} with $p, q$ replaced by $\tilde p, \tilde q$ of the same class.

\begin{thm} \label{thm:uniq5}
Suppose that $\Lambda = \tilde \Lambda$, $\Lambda_{jk} = \tilde \Lambda_{jk}$, $j,k = 1, 2$ (counting with multiplicities). Then $p(x) = \tilde p(x)$ and $q(x) = \tilde q(x)$ a.e. on $(-1, 1)$. Thus, the solution of Inverse Problem~\ref{ip:5} is unique.
\end{thm}

\begin{proof}
According to the above discussion, the characteristic functions $\Delta(\la)$, $\Delta_{jk}(\la)$, $j,k = 1, 2$, can be uniquely constructed by the spectra $\Lambda$, $\Lambda_{jk}$, $j,k = 1, 2$, by formulas \eqref{prod}, where the constants $c$, $c_{jk}$, $j, k = 1, 2$, are determined by \eqref{formc}-\eqref{formcjk}. Hence, under the conditions of this theorem, $\Delta(\la) \equiv \tilde \Delta(\la)$, $\Delta_{jk}(\la) \equiv \tilde \Delta_{jk}(\la)$, $j, k = 1, 2$. Using \eqref{UMU} and \eqref{MD}, we get $M(\la) \equiv \tilde M(\la)$. Therefore, Theorem~\ref{thm:uniqm} implies $Q(x) = \tilde Q(x)$ a.e. on $(0, 1)$. Taking the change of variables \eqref{changeZ} and \eqref{change} into account, we arrive at the assertion of the theorem.
\end{proof}

\begin{remark}
For clarity of exposition, we consider the simplest Dirichlet boundary conditions $u(-1) = u(1) = 0$. Analogously, our approach can be applied to other types of boundary conditions.
\end{remark}

\begin{remark}
Note that the condition $\al \in (-1, 1) \cup (\mathbb C \setminus \mathbb R)$ is equivalent to $\arg w_1 \ne \arg w_2$, $w_1 w_2 \ne 0$. This condition guarantees the existence of the four rays $\arg \rho = \theta_{j_s}$, $s = \overline{1, 4}$, with the property \eqref{ray}, which is crucial for our proofs. In the case $\al \in (-\iy, -1] \cup [1, \iy)$, the development of other methods is necessary for investigation of inverse problems.
\end{remark}

\begin{remark}
The functional-differential equation with involution
\begin{equation} \label{inv2}
-u''(x) + p(x) u(x) + q(x) u(-x) = \la u(x), \quad x \in (0, 1),
\end{equation}
considered, e.g., in \cite{Pol20}, can be reduced to the matrix form \eqref{eqvZ} with $\mathcal W = I$. The uniqueness of the reconstruction of the matrix Sturm-Liouville operator with $\mathcal W = I$ and the general self-adjoint boundary conditions has been proved in \cite{Bond21}. Although in \cite{Bond21} only the Hermitian potentials $Q(x) = Q^{\dagger}(x)$ are considered, this restriction is not necessary for the uniqueness of recovering the potential $Q(x)$ from the Weyl matrix. The results of \cite{Bond21} imply that the functions $p(x)$ and $q(x)$ in equation~\eqref{inv2} are uniquely specified by the five characteristic functions $\Delta(\la)$, $\Delta_{jk}(\la)$, $j, k = 1,2$, of the corresponding eigenvalue problems $\mathcal L$, $\mathcal L_{jk}$, $j, k = 1, 2$, for equation \eqref{inv2} instead of \eqref{eqvu} with the same boundary conditions as defined above. However, the problems $\mathcal L_{12}$ and $\mathcal L_{21}$ are irregular in this case. In particular, $\Delta_{12}(\la) \equiv \Delta_{21}(\la) \equiv 0$ if $p(x) = q(x) = 0$ in $(-1, 1)$. Consequently, only $\Delta(\la)$, $\Delta_{11}(\la)$, and $\Delta_{22}(\la)$ can be uniquely constructed by their zeros as infinite products. For $\Delta_{12}(\la)$ and $\Delta_{21}(\la)$, the constants $c_{12}$ and $c_{21}$ in \eqref{prod} should be additionally given.
\end{remark}

\begin{remark}
In this paper, we do not rigorously prove the minimality of the given spectral data. However, for the matrix Sturm-Liouville operators with $W = I$, the complete spectral data characterization is obtained in \cite{CK09, MT09, Bond19, Bond21-char}. Those spectral data are equivalent to the Weyl matrix. The known results imply that, for the non-self-adjoint potential $Q(x) \ne Q^{\dagger}(x)$, a certain small perturbation of one element of the Weyl matrix leads to the correct Weyl matrix corresponding to another potential $\tilde Q(x)$. Thus, the Weyl matrix is the minimal data for the unique determination of the matrix Sturm-Liouville operator. The author guesses that, in the case of weight $W \ne I$, the situation is similar, so the described five spectra are the minimal spectral data for the recovery of $p(x)$ and $q(x)$ in \eqref{eqvu}. 
\end{remark}

\section*{Appendix}
\setcounter{section}{5}
\setcounter{equation}{0}
\setcounter{thm}{0}

In Appendix, we transform the matrix Sturm-Liouville equation with a weight $W = \diag\{ w_j \}_{j = 1}^m$, $w_j \ne 0$, $w_j \ne w_k$ for $j \ne k$,
into the first-order system. 
This transform can be useful for investigation of various issues of the spectral theory of the matrix Sturm-Liouville operators and of FDO with involution.
We start with an auxiliary lemma.

\begin{lem} \label{lem:pos}
There exists $\la_* \in \mathbb C$ such that equation \eqref{eqv} with $\la = \la_*$ has an $(m \times m)$ matrix solution $X(x)$ such that $\det X(x) \ne 0$ for all $x \in [0, 1]$.
\end{lem}

\begin{proof}
Fix a sector $\Gamma = \Gamma_j$ and a ray $\{ \rho \colon \arg \rho = \varphi \} \subset \Gamma$. Let us prove that there exists $\rho_* > 0$ such that $\det C(x, \rho^2) \ne 0$ for all $\rho$ satisfying $|\rho| \ge \rho_*$, $\arg \rho = \varphi$ and for all $x \in [0, 1]$. The asymptotics \eqref{asymptE} and \eqref{asymptC} imply
$$
C(x, \rho^2) = \frac{1}{2}[I] (\exp(i \rho D x)[I] + \exp(-i\rho dx) I + O(\rho^{-1})), \quad \arg \rho = \varphi, \quad |\rho| \to \iy.
$$
Here and below in this proof, the convergence is uniform with 
with respect to $x \in [0, 1]$.
Let $a \in \mathbb C^m$ be an arbitrary unit vector: $\| a \| = \sqrt{\sum\limits_{j = 1}^m |a_j|^2} = 1$, and let
$$
b(x, \rho) = \frac{1}{2} (\exp(i \rho D x)[I] + \exp(-i\rho Dx)) a.
$$
We aim to prove that $\| b \| \ge c_0 > 0$ for sufficiently large $|\rho|$, $\arg \rho = \varphi$, and $x \in [0, 1]$. Clearly, $b = [b_j(x, \rho)]_{j = \overline{1, m}}^T$, where
\begin{equation} \label{defb}
b_j(x, \rho) = \cos (\rho d_j x) a_j + O(\rho^{-1} \exp(|\mbox{Im}\,(\rho d_j x)|), \quad \arg \rho = \varphi, \quad j = \overline{1, m}.
\end{equation}
Recall that 
$$
\mbox{Re}(i \rho d_j) = |\mbox{Im}(\rho d_j)| > 0, \quad \rho \in \Gamma, \quad j = \overline{1, m}.
$$
Therefore, the standard estimate yields
\begin{equation} \label{estc}
|\cos (\rho d_j x)| \ge c_1 \exp(|\mbox{Im}\,(\rho d_j x)|), \quad \arg \rho = \varphi, \quad j = \overline{1, m},
\end{equation}
where the constant $c_1$ does not depend on $|\rho|$ and $x$. Combining \eqref{defb} and \eqref{estc}, we conclude that, for every $\eps > 0$, there exists $\rho_{\eps} > 0$ such that 
$$
|b_j(x, \rho)| \ge c_1 |a_j| - \eps, \quad |\rho| \ge \rho_{\eps}, \quad \arg \rho = \varphi.
$$
Since $\| a  \| = 1$, then $\| b \| \ge c_1 > 0$ for sufficiently large $|\rho|$. Consequently, 
$$
\| C(x, \rho^2) a \| = \| [I] (b + O(\rho^{-1}) \| \ge c_2 > 0, \quad |\rho| \ge \rho_*, \quad \arg \rho = \varphi, \quad x \in [0, 1],
$$
for some $\rho_* > 0$. Hence, $\la_* = (\rho_* \exp(i \varphi))^2$, $X(x) = C(x, \la_*)$ satisfy the assertion of the lemma.
\end{proof}

It will be convenient for us to consider the equation of form
\begin{equation} \label{eqv2}
- \hat W Y''(x) + \hat Q(x) Y(x) = \la Y(x), \quad x \in (0, 1),
\end{equation}
where $\hat W = \diag \{ \hat w_k \}_{k = 1}^m$.
Clearly, equation~\eqref{eqv} can be easily reduced to \eqref{eqv2} by the transform $\hat W = W^{-1}$, $\hat Q(x) = W^{-1} Q(x)$.

Put $\hat D := \diag \{ \hat d_j \}_{j = 1}^m$, $d_j := \sqrt w_j$, $\mu = \sqrt{\la - \la_*}$ (the square root branch can be chosen arbitrarily), $U(x) := \hat D X'(x) X^{-1}(x)$, where $\la_*$ and $X(x)$ satisfy the assertion of Lemma~\ref{lem:pos}. Since the entries of $Q(x)$ belong to $L_1(0, 1)$, then $X'(x)$ is absolutely continuous on $[0, 1]$, and so does $U(x)$. Denote 
\begin{equation} \label{defY12}
Y_1(x) = Y(x), \quad Y_2(x) = -\mu^{-1}(\hat D Y_1'(x) - U(x) Y_1(x)). 
\end{equation}

\begin{lem}
For $\la \ne \la_*$ ($\mu \ne 0$), equation \eqref{eqv2} is equivalent to the system
\begin{align} \label{Y1}
    -\hat D Y_1' + U Y_1 & = \mu Y_2, \\ \label{Y2}
    \hat D Y_2' + \hat D U \hat D^{-1} Y_2 & = \mu Y_1.
\end{align}
\end{lem}

\begin{proof}
Relation \eqref{Y1} holds due to \eqref{defY12}. In order to prove the equivalence of \eqref{eqv2} and \eqref{Y1}-\eqref{Y2}, we substitute the definition \eqref{defY12} of $Y_2$ into the left-hand side of \eqref{Y2}:
\begin{equation} \label{sm1}
\hat D Y_2' + \hat D U \hat D^{-1} Y_2 =
\frac{1}{\mu} \hat D (-\hat D Y_1'' + U' Y_1  + U \hat D^{-1} U Y_1).
\end{equation}
Calculations show that
\begin{equation} \label{sm2}
U' = \hat D X'' X^{-1} - \hat D X' X^{-1} X' X^{-1} = \hat D^{-1} (\hat Q - \la_* I) - U \hat D^{-1} U.
\end{equation}
Substituting \eqref{sm2} into \eqref{sm1} and using \eqref{eqv2}, we arrive at \eqref{Y2}:
$$
\hat D Y_2' + \hat D U \hat D^{-1} Y_2 = \frac{1}{\mu} (- \hat W Y_1'' + \hat Q Y_1 - \la_* Y_1) = \mu Y_1.
$$
\end{proof}

The system \eqref{Y1}-\eqref{Y2} can be represented in the following block-matrix form:
$$
\begin{bmatrix} 0 & \hat D \\ -\hat D & 0 \end{bmatrix} \begin{bmatrix} Y_1' \\ Y_2' \end{bmatrix} + \begin{bmatrix} 0 & \hat D U \hat D^{-1} \\ U & 0 \end{bmatrix} \begin{bmatrix} Y_1 \\ Y_2 \end{bmatrix} = \mu \begin{bmatrix} Y_1 \\ Y_2 \end{bmatrix}.
$$
Diagonalizing the matrix $\begin{bmatrix} 0 & \hat D \\ -\hat D & 0 \end{bmatrix}$ by a suitable unitary transform
$$
\mathcal U^{\dagger} \begin{bmatrix} 0 & \hat D \\ -\hat D & 0 \end{bmatrix} \mathcal U = \begin{bmatrix} \hat D & 0 \\ 0 & -\hat D \end{bmatrix} =: \mathscr Q_0, \quad \mathcal U^{\dagger} = \mathcal U^{-1} \in \mathbb C^{2m \times 2m},
$$
we reduce the system \eqref{Y1}-\eqref{Y2} to the form
\begin{equation} \label{sys}
\mathscr Q_0 \mathscr Y'(x) + \mathscr Q(x) \mathscr Y(x) = \mu \mathscr Y(x),
\end{equation}
where
$$
\mathscr Y(x) := \mathcal U^{\dagger} \begin{bmatrix} Y_1(x) \\ Y_2(x) \end{bmatrix}, \quad \mathscr Q(x) = \mathcal U^{\dagger} \begin{bmatrix} 0 & \hat D U(x) \hat D^{-1} \\ U(x) & 0 \end{bmatrix} \mathcal U.
$$

Note that the described reduction of the matrix Sturm-Liouville equation \eqref{eqv} to form \eqref{sys} is non-unique, since the choice of the solution $X(x)$ satisfying Lemma~\ref{lem:pos} is non-unique. Hence, the potential $\mathscr Q(x)$ of the system \eqref{sys} can not be uniquely determined by any spectral data corresponding to \eqref{eqv}. Therefore, it is inconvenient to use the reduction to the first-order system to deal with inverse spectral problems for FDO with involution. Anyway, this reduction may be useful for studying direct spectral problems.

It is also worth mentioning that the inverse problem statement for the first-order system \eqref{sys} from \cite{Yur05-1, Yur05-2} has the principal difference comparing with the inverse problems studied in this paper. In \cite{Yur05-1, Yur05-2}, the linear forms $U_{\xi}(Y) = h_{\xi} Y(0)$ and $V_{\xi}(Y) = H_{\xi} Y(1)$ are introduced for $\xi = \overline{1, 2m}$, where $h_{\xi}$ and $H_{\xi}$ are row vectors of length $2m$, and the Weyl solutions $\Phi_k(x, \mu)$, $k = \overline{1, 2m}$, of \eqref{sys} are defined by the boundary conditions
$$
U_{\xi}(\Phi_k) = 0, \quad \xi = \overline{1, k-1}, \qquad U_k(\Phi_k) = 1, \qquad
V_{\eta}(\Phi_k) = 0, \quad \eta = \overline{1, 2m-k}. 
$$
This special structure of the boundary conditions allows the author of \cite{Yur05-1, Yur05-2} to achieve the triangular structure of the Weyl matrix. However, studying FDO with involution-reflection, we have the fixed boundary conditions induced by the continuity of $u(x)$ and its derivatives at $x = 0$, and these conditions cannot be changed. Therefore, Yurko's problem statement with the triangular structure appears to be unnatural for investigation of FDO with involution. The main technical feature of our problem is that suitable asymptotics are only valid on certain rays $\arg \rho = \theta_{j_s}$, $s = \overline{1, 4}$. This feature implies qualitative differences of our problem from the ones for the higher-order differential operators \cite{Yur02} and for the first-order systems \cite{Yur05-1, Yur05-2}.

\medskip

\textbf{Acknowledgement.} The author is grateful to Professor Sergey Buterin, who noticed that inverse problems for FDO with involution can be reduced to inverse problems for differential systems and investigated by developing the method of spectral mappings, and who encouraged the author to study such problems.

\noindent Natalia Pavlovna Bondarenko \\
1. Department of Applied Mathematics and Physics, Samara National Research University, \\
Moskovskoye Shosse 34, Samara 443086, Russia, \\
2. Department of Mechanics and Mathematics, Saratov State University, \\
Astrakhanskaya 83, Saratov 410012, Russia, \\
e-mail: {\it BondarenkoNP@info.sgu.ru}

\end{document}